\documentclass[a4paper,11pt]{article}
\usepackage{bbding,textcomp,amsfonts,amsthm,amsmath,mathrsfs}
\usepackage{amssymb}
\usepackage{multirow, url}
\usepackage[utf8x]{inputenc}
\usepackage{comment}
\usepackage[affil-it]{authblk}
\usepackage{titling}
\usepackage[T1]{fontenc}
\usepackage{ae,aecompl}
\usepackage{enumerate}

\newtheorem{definition}{Definition}
\newtheorem{theorem}{Theorem}
\newtheorem{conjecture}[theorem]{Conjecture}

\newtheorem{lemma}[theorem]{Lemma}

\newtheorem{case}{Case}
\newtheorem{subcase}{Case}
\numberwithin{subcase}{case}

\PrerenderUnicode{\unichar{355}}

\title{A new lower bound for the Towers of Hanoi problem}

\author{Codru\unichar{355} Grosu\thanks{This research was supported by the Deutsche Forschungsgemeinschaft within the research training group `Methods for Discrete Structures' (GRK 1408).}}
\affil{\small{grosu.codrut@gmail.com, Freie Universit\"at Berlin, Germany}}
\date{}

\begin{document}

\maketitle

\abstract{More than a century after its proposal, the Towers of Hanoi puzzle with $4$ pegs was solved by Thierry Bousch in a breakthrough paper in $2014$. The general problem with $p$ pegs is still open, with the best lower bound on the minimum number of moves due to Chen and Shen. We use some of Bousch's new ideas to obtain an asymptotic improvement on this bound for all $p \geq 5$.}

\section{\normalsize Introduction}

\textit{The Towers of Hanoi} is a puzzle invented by the French mathematician \'Edouard Lucas in $1883$ (\cite{Lucas1893}). The setup consists of $3$ pegs and $N$ disks of different sizes, arranged on the first peg in increasing order according to size. The goal is to move the disks from the first peg to another in as few moves as possible, such that the following three rules are always obeyed:
\begin{itemize}
\item[(R1)] only one disk can be moved at a time;
\item[(R2)] each move consists of taking the topmost disk on a peg and placing it on another peg;
\item[(R3)] a smaller disk is always moved on top of a larger one, or on an empty peg.
\end{itemize}
It is easy to see that the solution requires $2^N-1$ moves. The puzzle is very popular, and is frequently used to teach recursive algorithms to first-year computer science students.

Several variations of the original problem have been proposed (\cite{Stockmeyer94}), with one possibility being to increase the number of pegs available in the game. The puzzle with $4$ pegs was first introduced by Dudeney in $1908$ in his book \textit{The Canterbury Puzzles}, under the name "Reve's Puzzle". In $1939$, the general problem with $p$ pegs and $N$ disks was proposed in the \textit{American Mathematical Monthly} in the \textit{Advanced Problems} section, as Problem $3918$ (\cite{Stewart39}). Two years later, the journal published the proposer's (B.M. Stewart) claimed solution \cite{Stewart41}, as well as one solution submitted by a reader (J. S. Frame) \cite{Frame41}. The two solutions presented essentially equivalent formulas for the minimum number of moves needed, as well as an algorithm achieving the given bound. However, as noted by the Editors of the \textit{Monthly}, the two proofs rested on an unproven assumption about the optimality of the algorithm.

In fact, proving that the Frame-Stewart algorithm is best possible has since become a notorious open problem (\cite{Lunnon86}). However, in $2014$, more than a century after Dudeney's book appeared, the case $p=4$ was finally solved by Bousch (\cite{Bousch14}) in a very elegant way. We will say more about his beautiful solution later, but first let us describe the Frame-Stewart algorithm.

Given $N$ disks and $p$ pegs, the algorithm chooses an integer $1 \leq \ell < N$ that minimizes the number of steps in the following formula:
\begin{itemize}
\item Move the top $\ell$ disks from the start peg to an intermediate peg, using $p$ pegs.
\item Move the bottom $N-\ell$ disks from the start peg to the goal peg, using $p-1$ pegs (one peg is blocked by the $\ell$ smaller disks sitting on it).
\item Move the initial $\ell$ disks from the intermediate peg to the goal peg, using $p$ pegs.
\end{itemize}


Let $\Phi(p, N)$ denote the number of steps taken by the Frame-Stewart algorithm for $N$ disks and $p$ pegs. Then we have the recursive formula
\begin{equation}
\label{eq:FrameStewart}
\Phi(p, N) = \min_{1 \leq \ell < N} \left\{2\Phi(p, \ell) + \Phi(p-1, N-\ell)\right\},
\end{equation}
with initial data $\Phi(3, N) = 2^N-1$ and $\Phi(p, 1) = 1$.

Let $H(p, N)$ denote the minimum number of steps needed to move $N$ disks frome one peg to another, using $p$ pegs, according to the rules (R1)-(R3). We already know that $H(3, N) = \Phi(3, N)$. Building upon a result of Szegedy (\cite{Szegedy99}), Chen and Shen showed the following.
\begin{theorem}[Chen-Shen, \cite{Chen2004}]
\label{thm:chen}
For all $p \geq 3$ and $N \geq 1$ we have $H(p, N) \geq 2^{m - 1}$, where $m \geq 0$ is the largest integer such that $\binom{m+p-3}{p-2} < N$.
\end{theorem}
It is well-known that $\Phi(p, N) = \Theta(\frac{1}{(p-3)!}m^{p-3}2^m)$, and so by the above theorem $H(p, N)$ has the same growth rate as $\Phi(p, N)$. Theorem \ref{thm:chen} gives the best known lower bound on $H(p, N)$. Apart from this, Bousch has proved the following:
\begin{theorem}[Bousch, \cite{Bousch14}]
\label{thm:bousch}
For all $N \geq 1$ we have $H(4, N) = \Phi(4, N)$.
\end{theorem}

The main result of this note is the following asymptotic improvement of Theorem \ref{thm:chen}.
\begin{theorem}
\label{thm:main2}
Let $p \geq 4$ and $N \geq 1$. Write $N-1 = \binom{m+p-3}{p-2} + \binom{t+p-4}{p-3} + r$, with $m \geq t$ and $0 \leq r < \binom{t+p-4}{p-4}$ (this decomposition exists and is unique). Then we have $H(p, N) \geq (m+t)2^{m-2(p-2)}$.
\end{theorem}

The proof relies on the following idea, introduced by Szegedy. Rather than finding a lower bound for $H(p, N)$, one can try to bound the length $\Gamma(p, N)$ of the shortest sequence of steps that moves every disk at least once (here we also minimize over all possible starting configurations). Clearly $\Gamma(p, N)$ is then a lower bound for $H(p, N)$, as every disk must move at least once from the initial peg to the destination peg in the Hanoi problem. Szegedy has shown the following.
\begin{theorem}[Szegedy, \cite{Szegedy99}]
\label{thm:szegedy}
If $N \leq 1$ then $\Gamma(3, N) = N$. Otherwise $\Gamma(3, N) = 1 + 2^{N-2}$.
\end{theorem}

The main step in the proof of Theorem~\ref{thm:main2} is the following result, which may be of independent interest.
\begin{theorem}
\label{thm:main1}
For all $N \geq 0$ we have
\begin{equation*}
\Gamma(4, N) = \left\{
\begin{array}{ll}
N, & \textrm{if $N \leq 2$},\\
3+\frac{\Phi(4, N)-5}{4}, &\textrm{otherwise}.
\end{array}
\right.
\end{equation*}
\end{theorem}
In fact we believe that the following holds.
\begin{conjecture}
\label{conj:all}
For all $p \geq 3$ and $N \geq 0$ we have
\begin{equation*}
\Gamma(p, N) = \left\{
\begin{array}{ll}
N, & \textrm{if $N \leq p-2$},\\
p-1+\frac{\Phi(p, N)-(2(p-2)+1)}{4}, &\textrm{otherwise}.
\end{array}
\right.
\end{equation*}
\end{conjecture}
Theorems \ref{thm:szegedy} and \ref{thm:main1} show that Conjecture \ref{conj:all} holds for $p \in \{3, 4\}$.

\section{\normalsize Definitions and auxiliary results}

For $n \in \mathbb N$ let $[n] = \{0, 1, \ldots, n-1\}$ denote the set of natural numbers smaller than $n$. Given $p$ pegs and $N$ disks, we always label the pegs using numbers from $[p]$, and similarly the disks using numbers from $[N]$.

We now give a more precise description of $\Phi$ as follows.

\begin{definition}[The operators $\Delta_p$ and $\nabla_p$]
Let $p \geq 3$. We define for all $n \geq 0$ the values
\begin{equation*}
\Delta_p(n) := \binom{n+p-3}{p-2}
\end{equation*}
and
\begin{equation*}
\nabla_p(n) := \max\left\{ k \geq 0 : \Delta_p(k) \leq n\right\}.
\end{equation*}
\end{definition}

Note that $\nabla_p$ is well-defined, since $\Delta_p(0) = 0$. Then it can be shown (\cite{Klavzar02}, \cite{Rand09}) that for all $p \geq 3$ and $N \geq 1$,
\begin{equation}
\label{eq:PhiDef}
\Phi(p, N) = 2^{\nabla_p 0} + 2^{\nabla_p 1} + \ldots + 2^{\nabla_p(N-1)}.
\end{equation}
In the case $p=4$, this can be written more compactly as follows. Let $N-1 = \Delta_4 m + t, 0 \leq t \leq m$. Then
\begin{equation}
\label{eq:Phi4}
\Phi(4, N) = 1+(m+t)2^m.
\end{equation}
Note for later use the following property of $\Delta_p$:
\begin{equation}
\label{eq:delta}
\Delta_p n = \Delta_p (n-1) + \Delta_{p-1}n, \quad \forall p \geq 4, n \geq 1.
\end{equation} 

Let $p \geq 3$. We call an arrangement of disks on $p$ pegs a \textit{configuration} if no disk is placed on top of a larger one. Note that the set of configurations of $N$ disks can be identified with the set $[p]^{[N]}$ of functions $[N] \rightarrow [p]$, in particular, given a configuration $\mathbf{u}$, we let $\mathbf{u}^{-1}(x)$ denote the set of disks placed on peg $x$. Furthermore, $\mathbf{u}|_S$ represents the configuration obtained from $\mathbf{u}$ by deleting all disks in $[N] \setminus S$.

We define the \textit{Hanoi graph} $\mathcal{H}(p, N)$ as having vertex set $[p]^{[N]}$, and an edge between two vertices $u$ and $v$ if the corresponding configurations can be obtained from one another by a single disk move. We consider $\mathcal{H}(p, N)$ to be a metric space with the usual metric that has distance $1$ between any two adjacent vertices.

If $\gamma : [T] \rightarrow \mathcal{H}(p, N)$ is any path, we let $\ell(\gamma) := T-1$ denote its length. We sometimes write $\gamma_t$ instead of $\gamma(t)$, to denote the configuration at time $t$. For any $0 \leq t \leq T-2$, we let $D_{\gamma, t}$ be the unique disk moved between $\gamma(t)$ and $\gamma(t+1)$. We say that $D_{\gamma, t}$ \textit{is moved at time $t$}. For any $0 \leq t_1 \leq t_2 \leq T-1$, we let $\gamma|_{[t_1, t_2]}$ denote the path going through configurations $\gamma(t_1), \gamma(t_1+1), \ldots, \gamma(t_2)$.

The path $\gamma$ is called \textit{essential} if any disk is moved by $\gamma$ at least once. Note that in this case the path $\gamma^* : [T] \rightarrow  \mathcal{H}(p, N)$ given by $\gamma^*(t) := \gamma(T-t-1)$ is also essential. By definition,
$$\Gamma(p, N) := \min \{\ell(\gamma) : \gamma \textrm{ is an essential path in $\mathcal{H}(p, N)$}\}.$$

The structure of shortest paths (geodesics) in the Hanoi graph has been studied before (see \cite{Hinz14}). Note that an essential path need not be a geodesic.

We now introduce a crucial definition, due to Bousch. Let $E \subset \mathbb{N}$ be finite. For any $L \in \mathbb N$ we define
\begin{equation}
\label{eq:psiLDef}
\Psi_L(E) := (1-L)2^L - 1 + \sum_{n \in E}2^{\min\{\nabla_4 n, L\}},
\end{equation}
and further
\begin{equation}
\label{eq:psiDef}
\Psi(E) := \sup_{L \in \mathbb N} \Psi_L(E).
\end{equation}
The function $\Psi(E)$ is well-defined, as $\Psi_L(E)$ becomes negative for large $L$, and $\Psi_0(E) = |E|$. Bousch showed the following.
\begin{theorem}[Theorem $2.9$, \cite{Bousch14}]
\label{thm:mainBousch}
Let $a \in [4]$ arbitrary. Let $\mathbf{u}, \mathbf{v} \in \mathcal{H}(4, N)$ be two configurations such that in $\mathbf{v}$, peg $a$ and some other peg $b$ do not contain any disks. Then $d(\mathbf{u}, \mathbf{v}) \geq \Psi(\mathbf{u}^{-1}(a))$.
\end{theorem}
It turns out that $\Psi([N]) = \frac{\Phi(4, N+1)-1}{2}$. In combination with Theorem~\ref{thm:mainBousch}, this easily implies Theorem~\ref{thm:bousch}. We record this last fact below.
\begin{lemma}
\label{lem:obs3}
For all $N \geq 2$,
$$\Psi([N]) = \frac{\Phi(4, N+1) - 1}{2} = \min_{\substack{a+b = N\\a, b \geq 1}}\left\{\Phi(4, a) + \Phi(3, b)\right\}.$$
\end{lemma}
\begin{proof}
The first identity is Lemma $2.2$ from \cite{Bousch14}. The second follows from \eqref{eq:FrameStewart}.
\end{proof}
Lemma~\ref{lem:obs3} motivates the following definition. Let $N \geq 1$ and $\mathbf{u}$ be the configuration with all $N$ disks on peg $0$. A configuration $\mathbf{c}$ of $N$ disks on pegs $\{2, 3\}$ such that $d(\mathbf{u}, \mathbf{c}) \leq \frac{\Phi(4, N+1)-1}{2}$ is called \textit{a midpoint configuration of $N$ disks on $4$ pegs}. The existence of such configurations for all $N$ follows from the Frame-Stewart algorithm. Theorem~\ref{thm:mainBousch} shows that in fact $d(\mathbf{u}, \mathbf{c}) = \frac{\Phi(4, N+1)-1}{2}$ whenever $\mathbf{c}$ is a midpoint configuration, but we will not use this stronger statement.

We shall also need the following two lemmas.
\begin{lemma}[Lemma $2.6$, \cite{Bousch14}]
\label{lem:removal}
Let $A \subset \mathbb N$ finite, and $s$ a natural number such that $A - [\Delta_4 s]$ has at most $s$ elements. Then
$$\Psi(A) - \Psi(A - \{a\}) \leq 2^{s-1}$$
for all $a \in A$.
\end{lemma}
\begin{lemma}[Lemma $2.8$, \cite{Bousch14}]
\label{lem:union}
Let $A, B \subset \mathbb N$ be finite sets. Then
\begin{equation*}
\Psi(A) + \Psi(B) \geq \frac{\Phi(4, N+3) - 5}{4},
\end{equation*}
where $N : = |A \cup B|$.
\end{lemma}

Finally, we shall need the following recursive lower bound for $\Gamma$.
\begin{lemma}[Corollary $1$, \cite{Chen2004}]
\label{lem:recursive}
Let $p \geq 4$ and $N \geq 2$. Then for every $1 \leq \ell \leq N-1$,
\begin{equation*}
\Gamma(p, N) \geq 2\min\{\Gamma(p, N - \ell), \Gamma(p-1, \ell)\}.
\end{equation*}
\end{lemma}

\section{\normalsize The length of the shortest essential path}

We start with the following lemma.
\begin{lemma}
\label{lem:Two1}
Let $N \geq 1$ and $\mathbf{u}, \mathbf{v} \in \mathcal{H}(4, N)$ such that $\mathbf{u}^{-1}(\{2, 3\}) = \emptyset$ and $\mathbf{v}^{-1}(\{0, 1\}) = \emptyset$. Then
$$d(\mathbf{u}, \mathbf{v}) \geq 1 + \frac{\Phi(4, N+2) - 5}{4}.$$

Moreover, this inequality is tight.
\end{lemma}
\begin{proof}
Let $\gamma : [T] \rightarrow \mathcal{H}(4, N)$ be a shortest path between $\mathbf{u}$ and $\mathbf{v}$. Let $t_1 \in [T-1]$ be the first time when the disk $N-1$ moves to one of the pegs $2$ and $3$. Then we may assume without lack of generality that $\gamma_{t_1}(N-1) = 0$ and $\gamma_{t_1+1}(N-1) = 2$.

Set
\begin{align*}
A &= \{z \in [N-1] : \gamma_{t_1}(z) = 1\}\\
B &= \{z \in [N-1] : \gamma_{t_1}(z) = 3\}.
\end{align*}

Note that $A \dot\cup B = [N-1]$, as no other disk besides $N-1$ is on pegs $0$ or $2$ at time $t_1$. 

By Theorem~\ref{thm:mainBousch},
\begin{equation*}
d(\gamma(t_1), \gamma(0)) \geq \Psi(B),
\end{equation*}
as pegs $2$ and $3$ are empty in $\gamma(0) = \mathbf{u}$.

Similarly, by Theorem~\ref{thm:mainBousch} and the fact that all disks are placed on pegs $2$ and $3$ in $\mathbf{v}$,
\begin{equation*}
d(\gamma(t_1+1), \gamma(T-1)) \geq \Psi(A).
\end{equation*}

Consequently by Lemma~\ref{lem:union}, and the fact that the disk $N-1$ moves once at time $t_1$,
\begin{align*}
\ell(\gamma) &\geq d(\gamma(0), \gamma(t_1)) + 1 + d(\gamma(t_1+1), \gamma(T-1))\\
&\geq \Psi(B) + 1 + \Psi(A)\\
&\geq 1 + \frac{\Phi(4, N+2)-5}{2}.
\end{align*}

We now show that the inequality is tight. Let $a, b \in \mathbb N$ arbitrary such that $a+b = N$ and $b \geq 1$. Consider a configuration $\mathbf{u}_{a, b}$ with the disk $N-1$ on peg $1$, disks $N-b, N-b+1, \ldots, N-2$ on peg $0$, and disks $0, \ldots, a-1$ arranged on pegs $0$ and $1$ in such a way that they form a midpoint configuration of $a$ disks on $4$ pegs.

Then we can move the disks $0, \ldots, a-1$ to peg $3$ using at most $\frac{\Phi(4, a+1)-1}{2}$ moves.

Afterwards, we can move the disk $N-1$ to peg $2$.

Finally, we can move disks $N-b, \ldots, N-2$ to peg $2$ using $2^{b-1}-1$ moves.

Let $\mathbf{v}_{a, b}$ be the resulting configuration. It has disks $N-b, \ldots, N-1$ on peg $2$, and disks $0, \ldots, a-1$ on peg $3$. Also
\begin{equation*}
d(\mathbf{u}_{a, b}, \mathbf{v}_{a, b}) \leq \frac{\Phi(4, a+1)-1}{2} + 2^{b-1} \leq \frac{\Phi(4, a+1) + 2^b-1}{2} = \frac{\Phi(4, a+1) + \Phi(3, b)}{2}.
\end{equation*}
We now minimize over all choices of $a$ and $b$. This gives configurations $\mathbf{u}$ and $\mathbf{v}$ such that
\begin{align*}
d(\mathbf{u}, \mathbf{v}) &\leq \min_{\substack{a+b = N\\b \geq 1}} \frac{\Phi(4, a+1) + \Phi(3, b)}{2}\\
&=\frac{\Phi(4, N+2) - 1}{4}, \quad \textrm{by Lemma~\ref{lem:obs3},}\\
&= 1 + \frac{\Phi(4, N+2) - 5}{4}.
\end{align*}
\end{proof}
We would now like to extend this result to configurations which may share a peg, i.e. there is a peg which is occupied in both the starting and ending configuration. Surprisingly, this requires some more effort.
\begin{lemma}
\label{lem:thmext1}
Let $N \geq 1$ and $\mathbf{u}, \mathbf{v} \in \mathcal{H}(4, N)$ such that $\mathbf{u}^{-1}(\{2, 3\}) = \emptyset$ and $\mathbf{v}^{-1}(\{0, 3\}) = \emptyset$. If $\gamma : [T] \rightarrow \mathcal{H}(4, N)$ is any essential path between $\mathbf{u}$ and $\mathbf{v}$ then $\ell(\gamma) \geq \Psi(\mathbf{u}^{-1}(1))$.
\end{lemma}
\begin{proof}
We prove the lemma by induction on $N$.

If $N = 1$ then $\Psi([1]) = 1$ and the claim trivially holds.

So assume $N \geq 2$. By induction we may also assume that $N-1 \in E := \mathbf{u}^{-1}(1)$. Let $t_1$ be the first time when the disk $N-1$ moves. Then $\gamma_{t_1}(N-1) = 1$. Set $a := \gamma_{t_1+1}(N-1)$.

\begin{case}
$a \neq 2$.
\end{case}

Then $a \in \{0, 3\}$. Let $\pi$ be the involution on $\{0, 1, 2, 3\}$ which exchanges elements $1$ and $a$. We modify $\gamma$ into a new path $\gamma'$ by letting $\gamma'|_{[0, t_1+1]} = \gamma|_{[0, t_1+1]}$ and setting for all $t > t_1+1$,
\begin{align*}
\gamma'_t(D) &= \pi \circ \gamma_t(D),\, D \in [N-1],\\
\gamma_t'(N-1) &= a.
\end{align*}
At time $t_1+1$, peg $1$ is empty and peg $a$ only contains the disk $N-1$. Hence all moves represented by $\gamma'$ are valid moves. However, $\gamma'$ may contain repeated states, so we may need to delete some in order to make it into a proper path. Note that in $\gamma'(T-1)$ pegs $1$ and $3-a$ are empty, as in $\gamma(T-1)$ pegs $a$ and $3-a$ were empty. Consequently by Theorem~\ref{thm:mainBousch},
\begin{equation*}
\ell(\gamma) \geq \ell(\gamma') \geq d(\mathbf{u}, \gamma'(T-1)) \geq \Psi(\mathbf{u}^{-1}(1)).
\end{equation*}

\begin{case}
$a = 2$.
\end{case}

By Theorem~\ref{thm:mainBousch} and the fact that the pegs $1$ and $2$ are empty in $\gamma(t_1)|_{[N-1]}$, we have
\begin{equation*}
d(\gamma(0), \gamma(t_1)) \geq d(\gamma(0)|_{[N-1]}, \gamma(t_1)|_{[N-1]}) \geq \Psi(E - \{N-1\}).
\end{equation*}
Also, by Lemma~\ref{lem:Two1} and the fact that pegs $1$ and $2$ are empty in $\gamma(t_1+1)|_{[N-1]}$, while pegs $0$ and $3$ are empty in $\gamma(T-1)|_{[N-1]}$, we have
\begin{equation*}
d(\gamma(t_1+1), \gamma(T-1)) \geq d(\gamma(t_1+1)|_{[N-1]}, \gamma(T-1)|_{[N-1]}) \geq 1 + \frac{\Phi(4, N+1)-5}{4}.
\end{equation*}
Hence adding the move of the disk $N-1$ gives
\begin{equation*}
\ell(\gamma) \geq \Psi(E - \{N-1\}) + 1 + \frac{\Phi(4, N+1)-1}{4}.
\end{equation*}
Write $N = \Delta_4 m + t, 0 \leq t \leq m$. By Lemma~\ref{lem:removal} applied to $E$ and $s := m$, we get 
\begin{equation*}
\Psi(E) - \Psi(E - \{N-1\}) \leq 2^{m-1}.
\end{equation*}
Also $\Phi(4, N+1) = 1 + (m+t)2^m$ and so $\frac{\Phi(4, N+1) - 1}{4} = (m+t)2^{m-2}$.

If $N=2$ then $m=t=1$ and so $m+t \geq 2$.

If $N \geq 3$ then $\nabla_4 N = m \geq 2$ and again $m+t \geq 2$.

Thus in any case $m+t \geq 2$ and $(m+t)2^{m-2} \geq 2^{m-1}$. Hence
\begin{equation*}
\ell(\gamma) \geq \Psi(E) - 2^{m-1} + 1 + 2^{m-1} > \Psi(E) = \Psi(\mathbf{u}^{-1}(1)).
\end{equation*}
\end{proof}

\begin{lemma}
\label{lem:thmext2}
Let $N \geq 1$ and $\mathbf{u}, \mathbf{v} \in \mathcal{H}(4, N)$ such that $\mathbf{u}^{-1}(\{2, 3\}) = \mathbf{v}^{-1}(\{2, 3\}) = \emptyset$. If $\gamma : [T] \rightarrow \mathcal{H}(4, N)$ is any essential path between $\mathbf{u}$ and $\mathbf{v}$ then $\ell(\gamma) \geq \Psi(\mathbf{u}^{-1}(1))$.
\end{lemma}
The proof is nearly identical to that of Lemma~\ref{lem:thmext1}, and so we omit it.
\begin{lemma}
\label{lem:obs2}
If $N \geq 1$ then $\frac{\Phi(4, N+1)-1}{2} \geq \frac{\Phi(4, N+2) - 1}{4}$.
\end{lemma}
\begin{proof}
We show the equivalent statement $2\Phi(4, N+1) - 2 \geq \Phi(4, N+2) - 1$.

Write $N = \Delta_4 m +t, 0 \leq t \leq m$. Then $2\Phi(4, N+1) - 2 = (m+t)2^{m+1}$.

If $t \leq m-1$ then
\begin{equation*}
\Phi(4, N+2) - 1 = (m+t+1)2^m \leq 2(m+t)2^m,
\end{equation*}
as $N \geq 1$ implies $m+t \geq 1$.

If $t = m$ then
\begin{equation*}
\Phi(4, N+2) - 1 = (m+1)2^{m+1} \leq (m+t)2^{m+1}.
\end{equation*}
Thus in both cases $\Phi(4, N+2) - 1$ is at most $2\Phi(4, N+1) - 2$, as desired.
\end{proof}
We are now ready to prove the counterpart to Lemma~\ref{lem:Two1}.
\begin{lemma}
\label{lem:Two2}
Let $N \geq 1$ and $\mathbf{u}, \mathbf{v} \in \mathcal{H}(4, N)$ such that $\mathbf{u}^{-1}(\{2, 3\}) = \emptyset$ and $\mathbf{v}^{-1}(\{0, 3\}) = \emptyset$. If $\gamma : [T] \rightarrow \mathcal{H}(4, N)$ is any essential path between $\mathbf{u}$ and $\mathbf{v}$ then $$\ell(\gamma) \geq 1 + \frac{\Phi(4, N+2) - 5}{4}.$$
\end{lemma}
\begin{proof}
If $\mathbf{v}^{-1}(1) = \emptyset$ then all disks are on peg $2$ in $\mathbf{v}$. But pegs $2$ and $3$ are empty in $\mathbf{u}$, so by Theorem~\ref{thm:mainBousch} and Lemma~\ref{lem:obs3},
\begin{equation*}
\ell(\gamma) \geq d(\mathbf{v}, \mathbf{u}) \geq \Psi(\mathbf{v}^{-1}(2)) = \Psi([N]) = \frac{\Phi(4, N+1) - 1}{2}.
\end{equation*}
By Lemma~\ref{lem:obs2}, this is at least $1 + \frac{\Phi(4, N+2) - 5}{4}$, proving the claim in this case.

So we may assume that $\mathbf{v}^{-1}(1) \neq \emptyset$. Let $D$ be the largest disk on peg $1$ in $\mathbf{v}$.

\setcounter{case}{0}
\begin{case}
$D = N-1$.
\end{case}

Let $t_1$ be the last time when $D$ is not on peg $1$. Then $\gamma_{t_1+1}(D) = 1$. Set $a := \gamma_{t_1}(D)$. We define $b$ and $c$ as follows.
\begin{center}
\begin{tabular}{c | c c}
a & b & c\\
\hline
0 & 2 & 3\\
2 & 3 & 0\\
3 & 2 & 0
\end{tabular}
\end{center}

Then all disks in $[N-1]$ are on pegs $b$ and $c$ at time $t_1$. Set
\begin{align*}
A &:= \{z \in [N-1] : \gamma_{t_1}(z) = b\}\\
B &:= \{z \in [N-1] : \gamma_{t_1}(z) = c\}
\end{align*}

By Theorem~\ref{thm:mainBousch}, $d(\gamma(t_1), \mathbf{u}) \geq \Psi(A)$, as $b \in \{2, 3\}$ and pegs $2$ and $3$ are empty in $\mathbf{u}$.

Also, $d(\gamma(t_1+1), \mathbf{v}) \geq \Psi(B)$, as $c \in \{0, 3\}$ and pegs $0$ and $3$ are empty in $\mathbf{v}$.

Consequently by Lemma~\ref{lem:union} and the fact that $A \dot\cup B = [N-1]$,

$$\ell(\gamma) \geq d(\gamma(t_1), \mathbf{u}) + 1 + d(\gamma(t_1+1), \mathbf{v}) \geq 1 + \frac{\Phi(4, N+2) - 5}{4},$$
as desired.

\begin{case}
$D < N-1$.
\end{case}

Then $\mathbf{v}(N-1) = 2$. Let $t_1$ be the last time when the disk $N-1$ moves from pegs $\{0, 1\}$ to pegs $\{2, 3\}$. Let $t_2$ be the last time when $D$ is not on peg $1$.

Let $a := \gamma_{t_1+1}(N-1)$ and set
\begin{equation*}
b := \left\{
\begin{array}{cc}
2, & \textrm{if $a = 3$},\\
3, & \textrm{if $a = 2$}.
\end{array}
\right.
\end{equation*}
Define
\begin{equation*}
c := \left\{
\begin{array}{cc}
0, & \textrm{if $\gamma_{t_1}(N-1) = 1$},\\
1, & \textrm{if $\gamma_{t_1}(N-1) = 0$}.
\end{array}
\right.
\end{equation*}
Further set
\begin{align*}
A &:= \{z \in [N-1] : \gamma_{t_1}(z) = b\}\\
B &:= \{z \in [N-1] : \gamma_{t_1}(z) = c\}
\end{align*}
Note that $A \dot\cup B = [N-1]$. We now consider two subcases.
\begin{subcase}
$t_2 > t_1$.
\end{subcase}

Then
\begin{equation*}
d(\gamma(t_1), \mathbf{u}) \geq d(\gamma(t_1)|_{[N-1]}, \mathbf{u}|_{[N-1]}) \geq \Psi(A),
\end{equation*}
as $b \in \{2, 3\}$ and pegs $2$ and $3$ are empty in $\mathbf{u}$.

We will now show that $\ell(\gamma|_{[t_1+1, T-1]}) \geq \Psi(B)$.
 
If $c = 0$, then by Theorem~\ref{thm:mainBousch},
\begin{equation*}
\ell(\gamma|_{[t_1+1, T-1]}) \geq d(\gamma(t_1+1), \mathbf{v}) \geq d(\gamma(t_1+1)|_{[N-1]}, \mathbf{v}|_{[N-1]}) \geq \Psi(B),
\end{equation*}
as $c \in \{0, 3\}$ and pegs $0$ and $3$ are empty in $\mathbf{v}$.

If $c = 1$, then we claim that $\gamma|_{[t_1+1, T-1]}$ is an essential path when restricted to the moves of the first $N-1$ disks. Indeed, the disks on peg $1$ at time $t_1+1$ will all have to move, to make room for the disk $D$. The disks on peg $b$ at time $t_1+1$ will have to move, as either $b=3$ and $\mathbf{v}^{-1}(3) = \emptyset$, or $b=2$, and $N-1$ is not yet on peg $2$. Hence by Lemma~\ref{lem:thmext1}, if $b=3$, and Lemma~\ref{lem:thmext2}, if $b=2$,
\begin{equation*}
\ell(\gamma|_{[t_1+1, T-1]}) \geq \Psi(B).
\end{equation*}
So
\begin{equation*}
\ell(\gamma) \geq \Psi(A) + 1 + \Psi(B) \geq 1 + \frac{\Phi(4, N+2)-5}{4},
\end{equation*}
as desired.

\begin{subcase}
$t_2 < t_1$.
\end{subcase}

Then $c = 1$, otherwise $\gamma_{t_1}(N-1) = 1$ and $D$ is not on peg $1$ at time $t_1$. We claim that $\gamma|_{[0, t_1]}$ is an essential path when restricted to the moves of the first $N-1$ disks. Indeed, the disks on peg $1$ all moved at least once in the time interval $[0, t_1]$, as $D$ is already in final position at time $t_1$, and the disks on peg $b$ all moved, as $b \in \{2, 3\}$ and $\mathbf{u}^{-1}(\{2, 3\}) = \emptyset$. Hence by Lemma~\ref{lem:thmext1},
\begin{equation*}
\ell(\gamma|_{[0, t_1]}) \geq \Psi(B).
\end{equation*}
We will now show that
\begin{equation*}
\ell(\gamma|_{[t_1+1, T-1]}) \geq \Psi(A).
\end{equation*}
If $b = 3$ then the disk $N-1$ moves from peg $0$ to peg $2$ at time $t_1$. By Theorem~\ref{thm:mainBousch} and the fact that pegs $0$ and $3$ are empty in $\mathbf{v}$,
\begin{equation*}
\ell(\gamma|_{[t_1+1, T-1]}) \geq d(\gamma(t_1+1), \gamma(T-1)) \geq \Psi(A).
\end{equation*}
If $b=2$, let $t_3 > t_1$ be the last time when $\gamma_{t_3}(N-1) \neq 2$. Then at time $t_3$, peg $b = 2$ and some other peg do not contain any disks smaller than $N-1$. So by Theorem~\ref{thm:mainBousch},
\begin{equation*}
\ell(\gamma|_{[t_1+1, T-1]}) \geq d(\gamma(t_1+1)|_{[N-1]}, \gamma(t_3)|_{[N-1]}) \geq \Psi(A).
\end{equation*}
Thus in any case,
\begin{equation*}
\ell(\gamma) \geq \Psi(B)+1+\Psi(A) \geq 1 + \frac{\Phi(4, N+2) - 5}{4},
\end{equation*}
as claimed.
\end{proof}
\begin{proof}[Proof of Theorem~\ref{thm:main1}]
The statement is clearly true if $N \leq 3$, as $\Phi(4, 3) = 5$. So assume $N \geq 4$. We will first show the following inequality:
\begin{equation}
\label{eq:ToProve}
\Gamma(4, N) \geq 3 + \frac{\Phi(4, N)-5}{4}.
\end{equation}
Let $\gamma:[T] \rightarrow \mathcal{H}(4, N)$ be a shortest essential path. Let $t_1$ be any time when the disk $N-1$ moves. We may assume without lack of generality that $\gamma_{t_1}(N-1) = 0$ and $\gamma_{t_1+1}(N-1) = 1$. We shall further assume that $\gamma_{t_1+1}(N-2) = 2$.

We now choose $t_2 \in [T]$ such that the disk $N-2$ moves at time $t_2$, and the difference $|t_2-t_1|$ is minimal. Clearly $t_2$ exists, although there may be two distinct choices, if the disk $N-2$ moves before and after time $t_1$. If there are two possibilities for $t_2$, we choose one arbitrarily. Then by definition of $t_2$, the disk $N-2$ does not move in the time interval $[\min\{t_1, t_2+1\}, \max\{t_1+1, t_2\}]$.

Note that we can always replace $\gamma$ with $\gamma^*$, $t_1$ with $t_1' := T-t_1-2$ and $t_2$ with $t_2' := T-t_2-2$. Then $\gamma^*$ is still essential, $N-1$ moves at time $t_1'$, $N-2$ moves at time $t_2'$, and the difference $|t_1' - t_2'| = |t_1-t_2|$ is still minimal.

First suppose peg $3$ is empty at time $t_1+1$. By replacing $\gamma$ with $\gamma^*$ if necessary, we may assume that $t_2 > t_1$. At time $t_1+1$, all disks in $[N-1]$ are on peg $2$, while at time $t_2$, pegs $2$ and $\gamma_{t_2+1}(N-2)$ do not contain any disks from $[N-2]$. Consequently, by restricting to the first $N-2$ disks and applying Theorem~\ref{thm:mainBousch}, we get $\ell(\gamma|_{[t_1+1, t_2]}) \geq \frac{\Phi(4, N-1)-1}{2}$. Adding the $2$ moves of the disks $N-1$ and $N-2$ and using Lemma~\ref{lem:obs2}, we get
\begin{equation*}
\ell(\gamma) \geq 2 + \frac{\Phi(4, N-1)-1}{2} \geq 3 + \frac{\Phi(4, N)-5}{4}.
\end{equation*}
  
Therefore we may assume that peg $3$ is not empty at time $t_1+1$. Let $D$ be the largest disk on peg $3$ and $t_3$ any time when the disk $D$ moves.

\setcounter{case}{0}
\begin{case}
$t_2 < t_1$ but $t_3 > t_1$.
\end{case}

Let
\begin{align*}
A &:= \{z \in [N-2] :\gamma_{t_1+1}(z) = 2 \}\\
B &:= \{z \in [D] : \gamma_{t_1+1}(z) = 3 \}
\end{align*}

Let us look at the path $\gamma|_{[t_2+1, t_1]}$. If we go backwards from time $t_1$ to time $t_2+1$, all disks on peg $2$, except $N-2$ (in other words, the disks in $A$), will have to move to make room for the move of the disk $N-2$ at time $t_2+1$. Hence by restricting to the first $N-2$ disks and then applying Theorem~\ref{thm:mainBousch}, we get $\ell(\gamma|_{[t_2+1, t_1]}) \geq \Psi(A)$. Similarly, by restricting to the disks in $[D]$, we get $\ell(\gamma|_{[t_1+1, t_3]}) \geq \Psi(B)$. Adding the $3$ moves of the disks $N-1, N-2$ and $D$ gives
\begin{equation*}
\ell(\gamma) \geq 3 + \Psi(A) + \Psi(B) \geq 3 + \frac{\Phi(4, N)-5}{4}.
\end{equation*}

\begin{case}
$t_2 > t_1$ but $t_3 < t_1$.
\end{case}
This case follows from the previous one by replacing $\gamma$ with $\gamma^*$.

\begin{case}
$t_2, t_3 > t_1$ or $t_2, t_3 < t_1$.
\end{case}

By replacing $\gamma$ with $\gamma^*$ if necessary, we may suppose that $t_2, t_3 > t_1$. As the disk $N-2$ does not move in the time interval $[t_1, t_2]$, we have $\gamma_{t_2}(N-2) = 2$.

We shall consider two further subcases.

\begin{subcase}
$t_3 < t_2$.
\end{subcase}
Then $\gamma|_{[t_1+1, t_2]}$ is an essential path when restricted to the moves of the first $N-2$ disks. Indeed, all disks on peg $3$ must move, because $D$ moves, and all disks on peg $2$ move, to make room for the move of the disk $N-2$ at time $t_2$. Hence by Lemmas~\ref{lem:Two1} and~\ref{lem:Two2} applied to $\mathbf{u}:= \gamma(t_1+1)|_{[N-2]}$ and $\mathbf{v} := \gamma(t_2)|_{[N-2]}$, we get that
$$\ell(\gamma|_{[t_1+1, t_2]}) \geq 1+\frac{\Phi(4, N)-5}{4}.$$
Adding the further $2$ moves of the disks $N-1$ and $N-2$ gives the result.

\begin{subcase}
$t_3 > t_2$.
\end{subcase}

If $\gamma_{t_2+1}(N-2) = 3$ then the disk $D$ moves at least once in the time interval $[t_1+1, t_2]$ and we may apply the previous subcase.

So assume without lack of generality that $\gamma_{t_2+1}(N-2) = 0$ and $\gamma_{t_2+1}(D) = 3$. Set
\begin{align*}
A &:= \{z \in [N-2] : \gamma_{t_2+1}(z) = 1\}\\
B &:= \{z \in [D] : \gamma_{t_2+1}(z) = 3\}
\end{align*}
By Theorem~\ref{thm:mainBousch} and the fact that pegs $0$ and $1$ are empty in $\gamma(t_1+1)|_{[N-2]}$, we have $\ell(\gamma|_{[t_1+1, t_2]}) \geq \Psi(A)$. Similarly, by restricting to the moves of the disks in $[D]$, we see that $\ell(\gamma_{[t_2+1, t_3]}) \geq \Psi(B)$. Therefore by adding the $3$ moves of the disks $N-1, N-2$ and $D$ we get
\begin{equation*}
\ell(\gamma) \geq 3 + \Psi(A) + \Psi(B) \geq 3 + \frac{\Phi(4, N)-5}{4}.
\end{equation*}

This completes the proof of \eqref{eq:ToProve}.

We will now show that the bound can be achieved. Let $a, b \geq 0$ such that $a+b = N-3$. Consider a configuration $\mathbf{u}_{a, b}$ with the disk $N-1$ on peg $2$, the disk $N-2$ on peg $1$ and the disk $N-3$ on peg $0$. We put disks $N-3-b, N-2-b, \ldots, N-4$ on peg $0$, and distribute the remaining $a$ disks on pegs $0$ and $1$ in such a way that they form a midpoint configuration on $4$ pegs.

Then we can first move the disk $N-1$ to peg $3$, followed by the disks $0, 1, \ldots, a-1$ to the same peg in at most $\frac{\Phi(4, a+1)-1}{2}$ moves.

Afterwards we move the disk $N-2$ to peg $2$ (the peg is now free, and there are no more disks on top of the disk $N-2$). We further move the disks $N-3-b, N-2-b, \ldots, N-4$ to peg $2$ in $2^b - 1$ moves.

Finally, we move the disk $N-3$ to peg $1$.

Let $\mathbf{v}_{a, b}$ be the resulting configuration. We have just constructed an essential path $\gamma_{a, b}$ between $\mathbf{u}_{a, b}$ and $\mathbf{v}_{a, b}$ with 
\begin{equation*}
\ell(\gamma_{a, b}) \leq 2 + \frac{\Phi(4, a+1)-1}{2} + 2^b = 2 + \frac{\Phi(4, a+1)+\Phi(3, b+1)}{2}.
\end{equation*}
Minimizing over all choices of $a$ and $b$ yields an essential path $\gamma$ of length at most
\begin{align*}
2 + \min_{\substack{a+b = N-3\\a, b \geq 0}} \frac{\Phi(4, a+1)+\Phi(3, b+1)}{2}
&= 2 + \min_{\substack{a+b = N-1\\a, b \geq 1}} \frac{\Phi(4, a)+\Phi(3, b)}{2}\\
&= 2 + \frac{\Phi(4, N) - 1}{4}, \quad \textrm{by Lemma~\ref{lem:obs3}},\\
&= 3 + \frac{\Phi(4, N) - 5}{4}.
\end{align*}
\end{proof}

A similar argument as above shows that $\Gamma(p, N) \leq p-1 + \frac{\Phi(p, N)-(2(p-2)+1)}{4}$, for all $p \geq 3$ and $N \geq p-1$.

\section{\normalsize Proof of Theorem~\ref{thm:main2}}

Let us recall the statement of Theorem~\ref{thm:main2}. Given $p \geq 4$ and $N \geq 1$, we write
$$N-1 = \Delta_p m + \Delta_{p-1}t + r, \quad t \leq m,\,0 \leq r < \Delta_{p-2}(t+1).$$
Theorem~\ref{thm:main2} then states that $H(p, N) \geq (m+t)2^{m-2(p-2)}$.

Note that this decomposition of $N-1$ exists: first choose $m \geq 0$ maximal with $\Delta_p m \leq N-1$ and then $t \geq 0$ maximal with $\Delta_{p-1}t \leq N-1 - \Delta_p m$. Let $r$ be the remainder. Then $t \leq m$, as $\Delta_p m + \Delta_{p-1}(m+1) = \Delta_p(m+1)$. One can easily show that this decomposition is unique.
\begin{proof}[Proof of Theorem~\ref{thm:main2}]
We prove the stronger statement 
\begin{equation*}
\Gamma(p, N) \geq (m+t)2^{m-2(p-2)}
\end{equation*}
by induction, first after $p$, and then after $N$. The theorem then follows from the fact that $H(p, N) \geq \Gamma(p, N)$.

If $p = 4$, the claim reduces to the inequality
\begin{equation*}
\Gamma(4, N) \geq (m+t)2^{m-4}.
\end{equation*}
But by Theorem~\ref{thm:main1},
\begin{equation*}
\Gamma(4, N) = \left\{
\begin{array}{ll}
N, & \textrm{if $N \leq 2$},\\
2+(m+t)2^{m-2}, &\textrm{otherwise}.
\end{array}
\right.
\end{equation*}
As $0 = \Delta_4 0, 1 = \Delta_4 1$ and $2 = \Delta_4 1 + 1$, the claim holds in this case. So assume $p \geq 5$.

For $m \leq p-2$ the claim reduces to the inequality
\begin{equation*}
\Gamma(p, N) \geq \frac{m+t}{2^{2(p-2)-m}}.
\end{equation*}
But $\frac{m+t}{2^{2(p-2)-m}} \leq \frac{2(p-2)}{2^{p-2}} \leq 1$, as $2^x \geq 2x$ holds for all $x \geq 1$.
Hence in this case the claim is trivially true. So suppose $m \geq p-1$. Then we have two cases.

\setcounter{case}{0}
\begin{case}
$t+1 \leq m-1$.
\end{case}
Then $\Gamma(p, \Delta_p m + \Delta_{p-1} t + r + 1) \geq \Gamma(p, \Delta_p m + \Delta_{p-1} t +1)$.
But
\begin{equation*}
\Delta_p m + \Delta_{p-1} t + 1 = \Delta_p(m-1) + \Delta_{p-1}(m-1)+\Delta_{p-2}m + \Delta_{p-1}(t+1) - \Delta_{p-2}(t+1) + 1.
\end{equation*}
Furthermore $\Delta_{p-2}m \geq \Delta_{p-2}(m-1)+1$, with equality if $p = 5$. Also by our assumption $t+1 \leq m-1$ we have $\Delta_{p-2}(m-1) - \Delta_{p-2}(t+1) \geq 0$. Hence
\begin{equation*}
\Delta_p m + \Delta_{p-1}t+ 1 \geq \Delta_p(m-1) + \Delta_{p-1}(m-1) + \Delta_{p-1}(t+1) + 2.
\end{equation*}
So by Lemma~\ref{lem:recursive},
\begin{align*}
\Gamma(p, \Delta_p m + \Delta_{p-1} t + 1) \geq 2\min\{
&\Gamma(p, \Delta_p(m-1) + \Delta_{p-1}(t+1)+1),\\
&\Gamma(p-1, \Delta_{p-1}(m-1)+1)
\}.
\end{align*}
By induction,
\begin{equation*}
\Gamma(p, \Delta_p(m-1) + \Delta_{p-1}(t+1) + 1) \geq (m-1+t+1)2^{m-1-2(p-2)},
\end{equation*}
and
\begin{align*}
\Gamma(p-1, \Delta_{p-1}(m-1)+1)
&\geq (m-1)2^{m-1-2(p-1-2)}\\
&\geq 4(m-1)2^{m-1-2(p-2)}.
\end{align*}
As $4(m-1) = m + 3m-4 \geq m+t$, it follows that
\begin{equation*}
\Gamma(p, \Delta_p m + \Delta_{p-1} t + r + 1) \geq (m+t)2^{m-2(p-2)}.
\end{equation*}

\begin{case}
$m \geq t \geq m-1$.
\end{case}

Now
$$\Delta_p m +\Delta_{p-1} t +1 \geq \Delta_p m + 1 + \Delta_{p-1}(t-1) + \Delta_{p-2}(t-1) + 1.$$
Hence by Lemma~\ref{lem:recursive},
\begin{align*}
\Gamma(p, \Delta_p m + \Delta_{p-1} t + r + 1) \geq 2\min\{
&\Gamma(p, \Delta_p m + 1),\\
&\Gamma(p-1, \Delta_{p-1}(t-1) + \Delta_{p-2}(t-1) + 1)\}. 
\end{align*}
By induction,
\begin{equation*}
\Gamma(p, \Delta_p m + 1) \geq m2^{m-2(p-2)} = 2m2^{m-1-2(p-2)}\geq (m+t)2^{m-1-2(p-2)}.
\end{equation*}
Also,
\begin{equation*}
\Gamma(p-1, \Delta_{p-1}(t-1) + \Delta_{p-2}(t-1) + 1) \geq 2(t-1)2^{t-1-2(p-2)+2} \geq 4(t-1)2^{m-1-2(p-2)}.
\end{equation*}
As $4t-4 = t + 3t-4 \geq t + 3m-7 = m+t+2m-7$ and $m \geq p-1\geq 4 > \frac{7}{2}$, we get
\begin{equation*}
\Gamma(p-1, \Delta_{p-1}(t-1) + \Delta_{p-2}(t-1) + 1) \geq (m+t)2^{m-1-2(p-2)}.
\end{equation*}
Therefore
\begin{equation*}
\Gamma(p, \Delta_p m + \Delta_{p-1} t + r + 1) \geq (m+t)2^{m-2(p-2)}.
\end{equation*}
\end{proof}

\medskip
\textit{Acknowledgements.} I would like to thank Professor Thierry Bousch for answering my questions, for his comments and for spotting a mistake in an earlier version of this paper.

\bibliographystyle{abbrv}
\bibliography{bibl}
\end{document}